\documentclass[twoside]{article}
\usepackage{amssymb}
\usepackage{amsmath,amsthm}
\usepackage{mathrsfs}
\usepackage{amsbsy}
\usepackage{amstext}
\usepackage{arydshln}

\usepackage{blindtext} 

\usepackage[sc]{mathpazo} 
\usepackage[T1]{fontenc} 
\usepackage{microtype} 

\usepackage[english]{babel} 

\usepackage[hmarginratio=1:1,top=32mm,columnsep=20pt]{geometry} 
\usepackage[hang, small,labelfont=bf,up,textfont=it,up]{caption} 
\usepackage{booktabs} 

\usepackage{lettrine} 

\usepackage{enumitem} 
\setlist[itemize]{noitemsep} 

\usepackage{abstract} 

\usepackage{titlesec} 
\renewcommand\thesection{\Roman{section}} 
\renewcommand\thesubsection{\roman{subsection}} 
\titleformat{\section}[block]{\large\scshape\centering}{\thesection.}{1em}{} 
\titleformat{\subsection}[block]{\large}{\thesubsection.}{1em}{} 

\usepackage{fancyhdr} 
\usepackage{titling} 

\usepackage{hyperref} 

\pretitle{\begin{center}\Huge\bfseries} 
	\posttitle{\end{center}} 
\title{On a class of retarded integrodifferential equations}
\author{%
	\textsc{Fouad Maragh }\\
	\normalsize Laboratory LMA, Department of Mathematics,\\
	\normalsize Faculty of sciences,  Ibn Zohr University,  PB 80000 Agadir, Morocco. \\
	\normalsize \href{f.maragh@uiz.ac.ma}{f.maragh@uiz.ac.ma},  
}
\date{}
\renewcommand{\maketitlehookd}{%
	\begin{abstract}
The following class of retarded integro-differential equations in a Banach space
	\[
	\dot{x}\left(t\right)=Ax\left(t\right)+\int_{0}^{t}b\left(t-\tau\right)Lx_{\tau}d\tau+Kx_{t};\,\,t\geq0,
	\]
	 are taken into consideration in this study. The delay term $Lx_{\tau}$ of this equation is inserted into the integral as a convolution product with a scalar kernel. We prove the well-posedness of the problem under investigation using the Miyadera-Voigt perturbation and the theory of semigroups. We also explore the spectral analysis of an associated abstract Cauchy problem.
	
	\end{abstract}

	{\bf Keywords}: Integro-differential equations, $C_{0}$-semigroup, delay equations, Miyadera-Voigt perturbation.\\
	{\bf 2020 Mathematics Subject Classification} : 47D06, 45K05, 35P05, 34K08.}

\newtheorem{theorem}{Theorem}[section]

\newtheorem{definition}[theorem]{Definition}
\newtheorem{example}[theorem]{Example}
\newtheorem*{example*}{Example}
\newtheorem{lemma}[theorem]{Lemma}

\newtheorem{remark}[theorem]{Remark}

\begin{document}
\maketitle \markright{On a class of retarded integrodifferential equations} 

\section{Introduction.}  
  It is commonly acknowledged that equations in Banach spaces with delay present special challenges in terms of well-posedness. Delayed differential equations are a type of dynamical system that is significant for modelling hereditary phenomena in physics, biology, chemistry, economics, ecology, and other fields. Numerous articles have addressed the analysis of these equations using a semigroup method; for instance, J. Hale and G. Webb  \cite{hale1971functional,webb1974autonomous}  were among the pioneers in this area. We also cite 
  \cite{PritchardSalamon85,ClemmentDaprato1988,wu1996theory,rhandi1997extrapolation,kaiser2004integrated,batkai2001JMAA,batkai2005semigroups,HaddIdrissiRhandi2006,Calzadillas2010,hale2012theory} for references from more recent works.

In this paper, we focus on the following retarded integro-differential equation in a Banach space $X$
\begin{equation}
\begin{cases}
\dot{x}\left(t\right)=Ax\left(t\right)+\int_{0}^{t}b\left(t-\tau\right)Lx_{\tau}d\tau+Kx_{t}; & t\geq0,\\
x\left(0\right)=x,\,\,\,x_{0}=\varphi.
\end{cases}\label{RIE}
\end{equation}
Here, $A:D(A)\subset X\longrightarrow X$ generates a $C_{0}$-semigroup
$(\mathbb{T}(t))_{t\ge0}$ on $X$,  $b\left(\cdot\right)$ are scalar kernels in $W^{1,p}(\mathbb{R}_{+},\mathbb{C})$,
where $\mathbb{R}_{+}$ denotes the half line $\left[0,+\infty\right)$
and $1< p<+\infty$. The delay operator $L\in\mathcal{L}\left(\ensuremath{W^{1,p}}([-1,0],X);X\right)$,
the initial condition $\left(x,\varphi\right)\in X\times L^{p}\left([-1,0],X\right)$
and for each $t\ge0$, the history function $x_{t}:[-1,0]\to X$ of
$x\left(\cdot\right)$ is defined by $x_{t}(\theta)=x(t+\theta)$
for $\theta\in[-1,0]$. 

To the best of our knowledge, a semigroup technique has not yet been used to study this class of retarded equations where the delay term is in the convolution product.  The difficulty of this equation lies in the delay term $Lx_{t}$ which is one of the term of this convolution product $\int_{0}^{t}b\left(t-\tau\right)Lx_ {\tau}d\tau$, Knowing in advance that   $x_{t}$ satisfies the following equation
\begin{equation}
\begin{cases}
\frac{\partial}{\partial t}v\left(t,\theta\right)=\frac{\partial}{\partial\theta}v\left(t,\theta\right),\\
v\left(t,0\right)=x\left(t\right),\\
v\left(0,\theta\right)=\varphi\left(\theta\right).
\end{cases};\,\,t\geq0,\,-1\leq\theta\leq0.\label{eq:x_t}
\end{equation}

To overcome this difficulty, we were inspired by the paper \cite{DeschSchappacher1985}, by introducing the function $g\left(t,\cdot\right)\in L^{p}\left(\mathbb{R}_{+},X\right)$ defined by 

\begin{equation}
g\left(t\right)\left(s\right):=g\left(t,s\right)=\int_{0}^{t}b\left(t+s-\tau\right)Lx_{\tau}d\tau,\,\,t,s\geq0.\label{eq:Fct_g}
\end{equation}

This paper is organized as follows : In section 2, we rewrite the equation (\ref{RIE}) in associated product spaces
to an abstract Cauchy problem as follows 
$
\dot{\Upsilon}\left(t\right)=\mathcal{A}\Upsilon\left(t\right),
$
 and it is well known that this  system is well-posed if and only if $\mathcal{A}$ generates a $ C_0 $-semigroup. In our situation, the idea is to write $\mathcal{A}$ as a perturbation of a generator of a $C_{0}$-semigroup by an unbounded operator, then we apply \label{key}   Miyadera-Voigt perturbation   to obtain a sufficient condition  for  $\mathcal{A}$   to   generate  a $C_{0}$-semigroup.  In section 3, we study the spectral properties  of $ \mathcal{A} $,  by calculating the resolvent $R(\lambda ,  \mathcal{A})$ and the resolvent set $\rho( \mathcal{A})$ of the operator $ \mathcal{A } $. In section 4, we give two applications of the retarded integro-differential Volterra equations to illustrate the main results of this paper.

\section{ Well-posedness.}

In this section,  we denote by $X_{1}:=\left(D(A),\|\cdot\|_{1}\right)$
the Banach space equipped with the graph norm $\|x\|_{1}=\|x\|+\|Ax\|$,
and by $\rho(A)$ the resolvent set of $A,$ $\sigma(A)=\mathbb{C}\backslash\rho(A)$
the spectrum of $A$, $R(\lambda,A):=(\lambda I-A)^{-1}$ for $\lambda\in\rho(A)$
the resolvent operator of $A$, $\mathbb{C}_{\omega}=\left\{ \lambda\in\mathbb{C}\,\mid\:\Re\left(\lambda\right)>\omega\right\} $
with $\omega\in\mathbb{R}$, $ \Re\left(\lambda\right) $ the real part of $\lambda $ and let $Y$ be another Banach space,

\begin{definition} {\rm \cite{G.Weiss;1989b}}
	\label{def:MYVOGT} An operator $B\in\mathcal{L}(X_1,Y)$ is called
	an $q$-admissible observation operator for $\left(\mathbb{T}\left(t\right)\right)_{t\ge0}$, with $q\geq1$ if 
	\begin{equation}
	\int_{0}^{\tau}\|B\mathbb{T}(t)x\|^{q}dt\leq\gamma^{q}(\tau)\|x\|^{q}\label{adm_obs}
	\end{equation}
	for all $x\in D\left(A\right)$ and for some constants $\tau>0$ and
	$\gamma(\tau)>0$. \\
	In the special case when $Y=X$, $q=1$ and the
	constant $\gamma:=\gamma(\tau)<1$.
	\begin{equation}
	\int_{0}^{\tau}\|B\mathbb{T}(t)x\|dt\leq\gamma\|x\|,\label{MydrVgt}
	\end{equation}
	the operator $B$ is called Miyadera-Voigt
	perturbation for $\left(\mathbb{T}\left(t\right)\right)_{t\ge0}$.
\end{definition}

\begin{remark} \label{rem:qobservation} By Hölder\textquoteright s inequality it's
	easy to see that an $q$-admissible observation operator for $\left(\mathbb{T}\left(t\right)\right)_{t\ge0}$
	with $q>1$ from $\mathcal{L}(X_1,X)$ is also a Miyadera-Voigt perturbation.
\end{remark}

We will use in the sequel the following perturbation theorem (see.
\cite[Corollary 3.16, page 199]{EngelNagel;2000}).
\begin{theorem}[\cite{miyadera66,Voigt77}]
	\label{thm:MYVoGT} Let $A$ be the generator of $C_{0}$-semigroup
	$\left(\mathbb{T}\left(t\right)\right)_{t\ge0}$ on a Banach space
	$X$ and $B$ is a Miyadera-Voigt perturbation for $\left(\mathbb{T}\left(t\right)\right)_{t\ge0}$
	satisfy (\ref{MydrVgt}), then the sum $A+B$ with domain $D\left(A+B\right)=D\left(A\right),$
	generates a strongly continuous semigroup $\left(\mathscr{T}\left(t\right)\right){}_{t\geq0}$
	on $X$. Moreover $\left(\mathscr{T}\left(t\right)\right){}_{t\geq0}$
	satisfies 	
	$$ 	\mathscr{T}(t)x=\mathbb{T}(t)x+\int_{0}^{t}\mathbb{T}(t-s)B\mathscr{T}(s)xds $$
	and 
	$$ 	\int_{0}^{\tau}\|B\mathscr{T}(t)x\|dt\leq\frac{\gamma}{1-\gamma}\|x\| $$
	for $x\in D(A)$ and $t\geq0$. \\
\end{theorem}

Our current goal is to investigate the well-posedness of \eqref{RIE} with the
semigroup approach as given in \cite{batkai2001JMAA}. To do this
we will rewrite \eqref{RIE} as an abstract Cauchy problem $\dot{W}\left(t\right)=\mathcal{A}W\left(t\right)$
in an appropriate Banach space $\mathcal{X}$. To obtain this assertion, we use the function $ g $ defined in   \eqref{eq:Fct_g}, as a consequence, one has 
$$
\frac{\partial}{\partial t}g\left(t,s\right)=b\left(s\right)Lx_{t}+\frac{\partial}{\partial s}g\left(t,s\right).
$$
Taking into consideration the equation  (\ref{eq:x_t}), we can rewrite (\ref{RIE}) as follows 

$$
\begin{cases}
\dot{x}\left(t\right)=Ax\left(t\right)+Kx_{t}+g\left(t\right)\left(0\right)\\
\dot{x}_{t}=D_{\theta}x_{t}\\
\dot{g}\left(t\right)=D_{s}g\left(t\right)+b\left(\cdot\right)Lx_{t}\\
x\left(0\right)=x,\,x_{0}=\varphi\,\text{and}\,g\left(0\right)=0.
\end{cases};\,\,t\geq0,
$$
with 
\[
D_{\theta}=\frac{\partial}{\partial\theta};\,\,\,D\left(D_{\theta}\right)=\left\{ \varphi\in W^{1,p}\left([-1,0],X\right)\mid\varphi\left(0\right)=0\right\} 
\]
and 
\[
D_{s}=\frac{\partial}{\partial s};\,\,\,D\left(D_{s}\right)=W^{1,p}\left(\mathbb{R}_{+},X\right).
\]
Having chosen the product space 
$$
\mathcal{X}=X\times L^{p}\left(\left[-1,0\right],X\right)\times L^{p}\left(\mathbb{R}_{+},X\right),
$$
which is a Banach space with the norm 
$$
\left\Vert \begin{pmatrix}\chi\\
\psi\\
h
\end{pmatrix}\right\Vert =\|\chi\|+\|\psi\|_{L^{p}\left([-1,0],X\right)}+\|h\|_{L^{p}\left(\mathbb{R}_{+},X\right)},\quad\left(\chi,\psi,h\right)\in\mathcal{X},
$$

and set $\Upsilon\left(t\right)=\begin{pmatrix}x\left(t\right)\\
x_{t}\\
g\left(t\right)
\end{pmatrix}$, equation (\ref{RIE}) is rewritten in   $\mathcal{X}$  by the following Cauchy problem 
\begin{equation}
\begin{cases}
\dot{\Upsilon}\left(t\right)=\mathcal{A}\Upsilon\left(t\right)\\
\Upsilon\left(0\right)=\begin{pmatrix}x\\
\varphi\\
0
\end{pmatrix}
\end{cases};\,\,t\geq0,\label{ACP}
\end{equation}



with 

\begin{equation}   
\mathcal{A}=\begin{pmatrix}A & K & \delta_{0}\\
0 & D_{\theta} & 0\\
0 & b\left(\cdot\right)L & D_{s}
\end{pmatrix}\label{eq:Acali}
\end{equation}
and

\begin{equation} D\left(\mathcal{A}\right)=\left\{ \begin{pmatrix}\chi\\
\psi\\
h
\end{pmatrix}\in D\left(A\right)\times W^{1,p}\left([-1,0],X\right)\times W^{1,p}\left(\mathbb{R}_{+},X\right)\,;\,\psi\left(0\right)=\chi\right\} .
\label{eq:Dom-Acali} \end{equation}   

Here $\delta_{0}$ denotes the Dirac distribution, i.e., $\delta_{0}(f)=f(0)$
for each $f\in W^{1,p}\left(\mathbb{R}_{+},X\right)$. \\

When $\mathcal{A}$ generates a $C_{0}$-semigroup $(\mathscr{T}(t))_{t\ge0}$
on $\mathcal{X}$,   the mild solution of the abstract Cauchy problem
(\ref{ACP}) is written by 
$
\Upsilon \left(t\right)=\mathscr{T}(t)\Upsilon\left(0\right),\label{eq:AbstCauchySolut}
$
and it is a classical solution when $\Upsilon\left(0\right)\in D\left(\mathcal{A}\right)$. For more background on Semigroup theory we refer the reader to \cite{A.Pazy;1983,Curtain.Zwart;1995,EngelNagel;2000,staffans2005,TucsnakWeiss2009}. \\

Now, $\mathcal{A}$ can be expressed as a perturbation of the $C_{0}$-semigroup generator $\mathcal{A}_{0}$ by the operator $ \mathcal{B} $, as follows

\begin{equation}
\mathcal{A}=\mathcal{A}_{0}+\mathcal{B}
\end{equation}
with 
\[
\mathcal{A}_{0}=\left(\begin{array}{cc}
\boldsymbol{A} & \begin{matrix}0\\
0
\end{matrix}\\
\hdashline\begin{matrix}0 & 0\end{matrix} & D_{s}
\end{array}\right),\quad\mathcal{B}=\left(\begin{array}{cc}
\begin{pmatrix}0 & K\\
0 & 0
\end{pmatrix} & \begin{matrix}\delta_{0}\\
0
\end{matrix}\\
\hdashline\begin{matrix}0 & b\left(\cdot\right)L\end{matrix} & 0
\end{array}\right)
\]


and $ D\left(\mathcal{A}\right)=D\left(\mathcal{A}_{0}\right)=D\left(\mathcal{B}\right)=D\left(\boldsymbol{A}\right)\times W^{1,p}\left(\mathbb{R}_{+},X\right)$.\\
Here, $\boldsymbol{A}$ denotes  the operator $\begin{pmatrix}A & 0\\
0 & D_{\theta}
\end{pmatrix}$ with $D\left(\boldsymbol{A}\right)=\left\{ \begin{pmatrix}x\\
\varphi
\end{pmatrix}\in D\left(A\right)\times W^{1,p}\left([-1,0],X\right)\,;\,\varphi\left(0\right)=x\right\}.$
From \cite{batkai2001JMAA},   $\boldsymbol{A}$ generates
a $C_{0}$-semigroup $\left(\boldsymbol{T}\left(t\right)\right)_{t\geq0}$
on $\boldsymbol{X}=X\times L^{p}\left(\left[-1,0\right],X\right)$,
with 
$$
\boldsymbol{\boldsymbol{T}}\left(t\right)=\begin{pmatrix}\mathbb{T}(t) & 0\\
\mathbb{T}_{t} & \mathbb{S}^{0}\left(t\right)
\end{pmatrix},
$$

where $(\mathbb{S}^{0}\left(t\right))_{t\geq0}$ is the nilpotent left shift semigroup on $L^{p}\left(\left[-1,0\right],X\right)$ and $\mathbb{T}_{t}:X\rightarrow L^{p}\left(\left[-1,0\right],X\right)$
is defined by 
$$
(\mathbb{T}_{t}\,x)(\tau)=\begin{cases}
\mathbb{T}(t+\tau)x, & -t<\tau\leq0,\\
0, & -1\leq\tau\leq-t.
\end{cases}
$$
The diagonal of $\mathcal{A }_0  $ implies that $\mathcal{A }_0  $ generates the  $C_{0}$-semigroup 
$\left(\mathcal{T}_{0}\left(t\right)\right)_{t\geq0}$ defined by

$ \mathcal{T}_{0}\left(t\right)=\left(\begin{array}{c:c}
\boldsymbol{\boldsymbol{T}}\left(t\right) & \begin{matrix}0\\
0
\end{matrix}\\
\hdashline \begin{matrix}0 & 0\end{matrix} & \mathbb{S}\left(t\right)
\end{array}\right) $, here $\left(\mathbb{S}\left(t\right)f\right)(r)=f(t+r)$ denotes
the translation semigroup on $L^{p}\left(\mathbb{R}_{+},X\right)$.
\\

Throughout this paper, the following assumption will be required.
\begin{description}
	\item [{Assumption :}] 	A bounded operator $\varXi$ from $W^{1,p}([-1,0],X)$
	to $X$ satisfies the assumption $\left(M_{q}\right)$ for $ q\geq 1 $, if there exist,
	$\alpha>0$ and $\gamma_{\Xi}\left(\alpha\right)\geq 0$ such that 
	\[
	\left(M_{q}\right)\qquad\int_{0}^{\alpha}\|\varXi\left(\mathbb{T}_{\tau}\,x+\mathbb{S}^{0}\left(\tau\right)\varphi\right)\|^{q}\,d\tau\leq\gamma_{\varXi}^{q}\left(\alpha\right)\left\Vert \begin{pmatrix}x\\
	\varphi
	\end{pmatrix}\right\Vert ^{q}
	\]
	for all $\begin{pmatrix}x\\
	\varphi
	\end{pmatrix}\in D\left(\boldsymbol{A}\right)$.\\
	In the case when $ q=1 $, we assume also that  
	$\displaystyle\lim_{\alpha \to 0}  \gamma_{\Xi}\left(\alpha\right)=0$.	
	
\end{description}
In the following we will see two examples of operators verifying assumption $\left(M_{q}\right)$, and we refer to \cite{batkai2001JMAA} for many other interesting examples.

\begin{example}
\begin{enumerate}
	\item Let $\varXi:W^{1,2}\left([-1,0],X\right)\rightarrow X$, be the operator 
	defined by $\varXi=\delta_{-1}$, 	For $f\in D(A)$ and $\psi\in W^{1,2}([-1,0],X)$ satisfying $\psi(0)=f$, it is easy to show that there is $0<\alpha<1$ such that 
	\begin{align}
	\int_{0}^{\alpha}\|\varXi\left(\mathbb{T}_{\tau}\,f+\mathbb{S}^{0}\left(\tau\right)\psi\right)\|^{2}d\tau & =\int_{0}^{\alpha}\|\psi(\tau-1)\|^{2}\,d\tau,\nonumber \\
	& \leq\|\psi\|_{L^{2}\left(\left[-1,0\right],X\right)}^{2},\label{eq:myExmple}\\
	& \leq\left\Vert \begin{pmatrix}f\\
	\psi
	\end{pmatrix}\right\Vert ^{2}.\nonumber 
	\end{align}
	This implies that assumption $\left(M_{2}\right)$ is satisfied by $\varXi$.
	\item 	Let  $\eta:[-1,0]\rightarrow X$ be of bounded variation and  $\varXi:C([-1,0],X)\rightarrow X$ be the bounded linear
	operator given by the Riemann-Stieltjes integral 
	$ 	\varXi(\varphi)=\int_{-1}^{0}d\eta(\theta)\,\varphi(\theta) $ for all $\varphi\in C([-1,0],X).\label{Stieljes} $
	Since $W^{1,2}([-1,0],X)$ is continuously embedded in $C([-1,0],X)$, $\varXi$ defines a bounded operator from $W^{1,2}([-1,0],X)$ to $X$, and 
	\begin{equation}
	\int_{0}^{\alpha}\|\varXi\left(\mathbb{T}_{\tau}\,x+\mathbb{S}^{0}\left(\tau\right)\varphi\right)\|\,d\tau\leq\alpha^{\frac{1}{2}}M\,|\eta|([-1,0])\left\Vert \begin{pmatrix}x\\
	\varphi
	\end{pmatrix}\right\Vert, \label{cond.miyadera}
	\end{equation}
	for all $0<\alpha<1$, where  $M =\sup_{\tau\in[0,1]}\|\mathbb{T}(\tau)\|$
	and $|\eta|$ is the positive Borel measure on $[-1,0]$ defined by
	the total variation of $\eta$. It is clear that $\displaystyle\lim_{\alpha \to 0}  \gamma_{\varXi}\left(\alpha\right)=0$ with  $\gamma_{\varXi}\left(\alpha\right)=\alpha^{\frac{1}{2}}M\,|\eta|([-1,0])$. Then assumption
	$(M_{1})$ is satisfied.
\end{enumerate}	
\end{example}

The following Lemmas are necessary to establish the section's main result,

\begin{lemma}
	\label{lem:Mydvgt} If $L$ and $K$ satisfy assumption $(M_{1})$,
	then $\mathcal{B}$ is a Miyadera-Voigt perturbation for $\left(\mathcal{T}_{0}\left(t\right)\right)_{t\geq0}$.
\end{lemma}

\begin{proof}
	Let $\begin{pmatrix}\chi\\
	\psi\\
	h
	\end{pmatrix}\in D\left(\mathcal{A}\right)$, $L$ and $K$ satisfy assumption $(M_{1})$ then there is $\alpha>0$
	such that  
	
	\begin{align*}
		\int_{0}^{\alpha}\left\Vert \mathcal{B}\mathcal{T}_{0}\left(\tau\right)\begin{pmatrix}\chi\\
			\psi\\
			h
		\end{pmatrix}\right\Vert d\tau  = & \int_{0}^{\alpha}\left\Vert \begin{pmatrix}K\left(\mathbb{T}_{\tau}\,\chi+\mathbb{S}^{0}\left(\tau\right)\psi\right)+h\left(\tau\right)\\
			0\\
			b\left(\cdot\right)L\left(\mathbb{T}_{\tau}\,\chi+\mathbb{S}^{0}\left(\tau\right)\psi\right)
		\end{pmatrix}\right\Vert d\tau\\
		 \leq & \alpha^{1-\frac{1}{p}}\left\Vert h\right\Vert _{L^{p}\left(\mathbb{R}_{+},X\right)}+\int_{0}^{\alpha}\left\Vert K\left(\mathbb{T}_{\tau}\,\chi+\mathbb{S}^{0}\left(\tau\right)\psi\right)\right\Vert d\tau\\
		  & +\left\Vert b\right\Vert _{L^{p}\left(\mathbb{R}_{+}\right)}\int_{0}^{\alpha}\left\Vert L\left(\mathbb{T}_{\tau}\,\chi+\mathbb{S}^{0}\left(\tau\right)\psi\right)\right\Vert d\tau\\
		 \leq & \alpha^{1-\frac{1}{p}}\left\Vert h\right\Vert _{L^{p}\left(\mathbb{R}_{+},X\right)}+\gamma_{K}\left(\alpha\right)\left\Vert \begin{pmatrix}\chi\\
			\psi
		\end{pmatrix}\right\Vert +\gamma_{L}\left(\alpha\right)\left\Vert b\right\Vert _{L^{p}\left(\mathbb{R}_{+}\right)}\left\Vert \begin{pmatrix}\chi\\
			\psi
		\end{pmatrix}\right\Vert \\
		 \leq & \gamma\left(\alpha\right)\left\Vert \begin{pmatrix}\chi\\
			\psi\\
			h
		\end{pmatrix}\right\Vert 
	\end{align*}
	With $\gamma\left(\alpha\right)=\max\left(\gamma_{K}\left(\alpha\right);\gamma_{L}\left(\alpha\right)\left\Vert b\right\Vert _{L^{p}\left(\mathbb{R}_{+}\right)};\alpha^{1-\frac{1}{p}}\right)$,
	choose now $\alpha$ small enough such that $0\leq\gamma\left(\alpha\right)<1$,
	then $\mathcal{B}$ is a Miyadera-Voigt perturbation for $\left(\mathcal{T}_{0}\left(t\right)\right)_{t\geq0}$.
\end{proof}
\begin{lemma}
	\label{lem:qobserv} If $L$ and $K$ satisfy  assumption $(M_{p})$ then $\mathcal{B}$ is a $p$-admissible observation  operator
	for $\left(\mathcal{T}_{0}\left(t\right)\right)_{t\geq0}$.
\end{lemma}

\begin{proof}
	Let $\begin{pmatrix}\chi\\
	\psi\\
	h
	\end{pmatrix}\in D\left(\mathcal{A}\right)$, $L$ and $K$ satisfy the Assumption $(M_{p})$   then
	there is $\alpha>0$ such that 
	
	\begin{align*}
		\int_{0}^{\alpha}\left\Vert \mathcal{B}\mathcal{T}_{0}\left(\tau\right)\begin{pmatrix}\chi\\
			\psi\\
			h
		\end{pmatrix}\right\Vert ^{p}d\tau  = & \int_{0}^{\alpha}\left\Vert \begin{pmatrix}K\left(\mathbb{T}_{\tau}\,\chi+\mathbb{S}^{0}\left(\tau\right)\psi\right)+h\left(\tau\right)\\
			0\\
			b\left(\cdot\right)L\left(\mathbb{T}_{\tau}\,\chi+\mathbb{S}^{0}\left(\tau\right)\psi\right)
		\end{pmatrix}\right\Vert ^{p}d\tau\\
		 \leq & 4^{p}\left(\left\Vert h\right\Vert _{L^{p}\left(\mathbb{R}_{+},X\right)}^{p}+\int_{0}^{\alpha}\left\Vert K\left(\mathbb{T}_{\tau}\,\chi+\mathbb{S}^{0}\left(\tau\right)\psi\right)\right\Vert ^{p}d\tau\right.\\
		  & \left.+\left\Vert b\right\Vert _{L^{p}\left(\mathbb{R}_{+}\right)}^{p}\int_{0}^{\alpha}\left\Vert L\left(\mathbb{T}_{\tau}\,\chi+\mathbb{S}^{0}\left(\tau\right)\psi\right)\right\Vert ^{p}d\tau\right)\\
		 \leq & 4^{p}\left(\left\Vert h\right\Vert _{L^{p}\left(\mathbb{R}_{+},X\right)}^{p}+\gamma_{K}^{p}\left(\alpha\right)\left\Vert \begin{pmatrix}\chi\\
			\psi
		\end{pmatrix}\right\Vert ^{p}+\gamma_{L}^{p}\left(\alpha\right)\left\Vert b\right\Vert _{L^{p}\left(\mathbb{R}_{+}\right)}^{p}\left\Vert \begin{pmatrix}\chi\\
			\psi
		\end{pmatrix}\right\Vert ^{p}\right)\\
		 \leq & \gamma^{p}\left(\alpha\right)\left\Vert \begin{pmatrix}\chi\\
			\psi\\
			h
		\end{pmatrix}\right\Vert ^{p}.
	\end{align*}
	With $\gamma\left(\alpha\right)=4\times\max\left(1;\gamma_{K}\left(\alpha\right);\gamma_{L}\left(\alpha\right)\left\Vert b\right\Vert _{L^{p}\left(\mathbb{R}_{+}\right)}\right)$.
	Then $\mathcal{B}$ is a $p$-admissible observation  operator for $\left(\mathcal{T}_{0}\left(t\right)\right)_{t\geq0}$.
\end{proof}
The main result of this section is the following Theorem. 
\begin{theorem}
	\label{thm:Wellposedness} If $L$ and $K$ satisfy   conditions
	of one of Lemma \ref{lem:Mydvgt} and Lemma \ref{lem:qobserv}, then the equation \eqref{RIE}
	is well-posed and the solution is writing by 
	\begin{equation}
	x\left(t\right)=\mathbb{T}(t)x+\int_{0}^{t}\mathbb{T}(t-s)\int_{0}^{s}b\left(s-\tau\right)Lx_{\tau}d\tau ds+\int_{0}^{t}\mathbb{T}(t-s)Kx_{s}ds,\label{eq:MildSolut}
	\end{equation}
	for $\begin{pmatrix}x\\
	\varphi
	\end{pmatrix}\in D\left(A\right)\times W^{1,p}\left([-1,0],X\right)$.
\end{theorem}

\begin{proof}
	Let $\begin{pmatrix}x\\
	\varphi
	\end{pmatrix}\in D\left(A\right)\times W^{1,p}\left([-1,0],X\right)$, we now that $\mathcal{A}=\mathcal{A}_{0}+\mathcal{B}$ and  $ \mathcal{A}_{0} $ generates the following semigroup
	$$
	\mathcal{T}_{0}\left(t\right)=\begin{pmatrix}\mathbb{T}(t) & 0 & 0\\
	\mathbb{T}_{t} & \mathbb{S}^{0}\left(t\right) & 0\\
	0 & 0 & \mathbb{S}\left(t\right)
	\end{pmatrix},
	$$
	by  Lemma \ref{lem:Mydvgt} or Lemma \ref{lem:qobserv},  Remark \ref{rem:qobservation} and  applying Theorem \ref{thm:MYVoGT}, we conclude that $ \mathcal{A} $ generates a strongly continuous semigroup $ 	\left(\mathcal{T}\left(t\right)\right)_{t\geq 0} $ satisfying
	$$
	\mathcal{T}\left(t\right)\begin{pmatrix}x\\
	\varphi\\
	0
	\end{pmatrix}=\mathcal{T}_{0}\left(t\right)\begin{pmatrix}x\\
	\varphi\\
	0
	\end{pmatrix}+\int_{0}^{t}\mathcal{T}_{0}\left(t-s\right)\mathcal{B}\mathcal{T}\left(s\right)\begin{pmatrix}x\\
	\varphi\\
	0
	\end{pmatrix}ds
	$$
	and by equation (\ref{ACP}) we have 
	$$
	\begin{pmatrix}x\left(t\right)\\
	x_{t}\\
	g\left(t\right)
	\end{pmatrix}=\mathcal{T}_{0}\left(t\right)\begin{pmatrix}x\\
	\varphi\\
	0
	\end{pmatrix}+\int_{0}^{t}\mathcal{T}_{0}\left(t-s\right)\mathcal{B}\begin{pmatrix}x\left(s\right)\\
	x_{s}\\
	g\left(s\right)
	\end{pmatrix}ds
	$$
	which implies 
	\begin{align*}
		\begin{pmatrix}x\left(t\right)\\
			x_{t}\\
			g\left(t\right)
		\end{pmatrix}  = & \begin{pmatrix}\mathbb{T}(t)x\\
			\mathbb{T}_{t}x+\mathbb{S}^{0}\left(t\right)\varphi\\
			0
		\end{pmatrix}+\int_{0}^{t}\begin{pmatrix}\mathbb{T}(t-s)\left(Kx_{s}+\delta_{0}g\left(s\right)\right)\\
			\mathbb{T}_{t-s}\left(Kx_{s}+\delta_{0}g\left(s\right)\right)\\
			\mathbb{S}\left(t-s\right)b\left(\cdot\right)Lx_{s}
		\end{pmatrix}ds\\
		 = & \begin{pmatrix}\mathbb{T}(t)x+\int_{0}^{t}\mathbb{T}(t-s)\left(Kx_{s}+\delta_{0}g\left(s\right)\right)ds\\
			\mathbb{T}_{t}x+\mathbb{S}^{0}\left(t\right)\varphi+\int_{0}^{t}\mathbb{T}_{t-s}\left(Kx_{s}+\delta_{0}g\left(s\right)\right)ds\\
			\int_{0}^{t}\mathbb{S}\left(t-s\right)b\left(\cdot\right)Lx_{s}ds
		\end{pmatrix}.
	\end{align*}
	The solution is the first component given by 
	\begin{align*}
	x\left(t\right) & =\mathbb{T}(t)x+\int_{0}^{t}\mathbb{T}(t-s)\left(Kx_{s}+\delta_{0}g\left(s\right)\right)ds\\
	& =\mathbb{T}(t)x+\int_{0}^{t}\mathbb{T}(t-s)\int_{0}^{s}b\left(s-\tau\right)Lx_{\tau}d\tau ds+\int_{0}^{t}\mathbb{T}(t-s)Kx_{s}ds
	\end{align*}
	which ends the proof. 
\end{proof}

\section{Spectral analysis.}

In this section, we consider \eqref{RIE} is well-posed and  we calculate the resolvent set $\rho\left(\mathcal{A}\right)$
and resolvent operator $R\left(\lambda,\mathcal{A}\right)$ of the
generator $\mathcal{A}$ and we derive an interesting formula similar
that have been obtained in \cite[Proposition 3.19, page 56]{batkai2005semigroups}.
We will need throughout this section the Laplace transform of the scalar kernel $b(\cdot)$, for this we must assume that, There is $\beta \in \mathbb{R}$ such that $b(\cdot) e^{-\beta \cdot} \in L^{1}\left(\mathbb{R}^{+}\right)$ and we note this Laplace transform by $\hat{b}\left(\cdot\right)$ defined by
$$
\hat{b}(\lambda)=\int_{0}^{\infty} e^{-\lambda t} b(t) d t, \lambda \in \mathbb{C}_{\beta}.
$$ 
The main result of this section is the following Theorem. 
\begin{theorem}
	\label{thm:Spectr} Let $(\mathcal{A},D(\mathcal{A}))$ be the operator
	defined by (\ref{eq:Acali}) and (\ref{eq:Dom-Acali}). For $\lambda\in\mathbb{C}_{0}\cap\mathbb{C}_{\beta}\cap\rho(A)$
	we have 
	$$
	\lambda\in\rho\left(\mathcal{A}\right)\text{ \ensuremath{\iff} }1\in\rho\left(e_{\lambda}R\left(\lambda,A\right)\left(K+\hat{b}\left(\lambda\right)L\right)\right)\text{ \ensuremath{\iff} }\lambda\in\rho\left(A+\left(K+\hat{b}\left(\lambda\right)L\right)e_{\lambda}\right)
	$$
	where $\left(e_{\lambda}x\right)(\theta)=e^{\lambda\theta}x$ for
	$x\in X$ and $\theta\in[-1,0]$. 
	Moreover, for
	$\lambda\in\rho(\mathcal{A})$ the resolvent $$R(\lambda,\mathcal{A})=$$
	\[
	\left(\begin{array}{c:c:c}
	R_{\lambda} & R_{\lambda}\Sigma_{\lambda}R\left(\lambda,D_{\theta}\right) & \begin{array}{c}
	\left[R_{\lambda}\left(K_{\lambda}+L_{\lambda}\right)R\left(\lambda,A\right)\delta_{0}\right.\\
	\left.+R\left(\lambda,A\right)\delta_{0}\right]R\left(\lambda,D_{s}\right)
	\end{array}\\
	\hdashline e_{\lambda}R_{\lambda} & \begin{array}{c}
	\left[e_{\lambda}R_{\lambda}\Sigma_{\lambda}\right.\\
	\left.+I\right]R\left(\lambda,D_{\theta}\right)
	\end{array} & e_{\lambda}R_{\lambda}\delta_{0}R\left(\lambda,D_{s}\right)\\
	\hdashline R\left(\lambda,D_{s}\right)L_{\lambda}R_{\lambda}R\left(\lambda,D_{s}\right) & \begin{array}{c}
	R\left(\lambda,D_{s}\right)\left[b\left(\cdot\right)L\right.\\
	\left.+L_{\lambda}R_{\lambda}\Sigma_{\lambda}\right]R\left(\lambda,D_{\theta}\right)
	\end{array} & \begin{array}{c}
	R\left(\lambda,D_{s}\right)\left[I\right.\\
	\left.+b\left(\cdot\right)LR_{\lambda}\delta_{0}\right]R\left(\lambda,D_{s}\right)
	\end{array}
	\end{array}\right)
	\]
%
%
	with  
	$$
	\begin{cases}
	\Sigma_{\lambda}=K+\hat{b}\left(\lambda\right)L,\\
	R_{\lambda}=R\left(\lambda,A+\Sigma_{\lambda}e_{\lambda}\right),\\
	L_{\lambda}=\hat{b}\left(\lambda\right)Le_{\lambda},\\
	K_{\lambda}=Ke_{\lambda}.
	\end{cases}
	$$
	
\end{theorem}

\begin{proof}
	Let $\lambda\in\mathbb{C}_{0}\cap\rho(A)$, we have  
	\begin{align*}
	R\left(\lambda,\mathcal{A}_{0}\right) & = \int_{0}^{+\infty}e^{-\lambda t}\mathcal{T}_{0}\left(t\right)dt\\
	& =  \begin{pmatrix}R\left(\lambda,A\right) & 0 & 0\\
	e_{\lambda}R\left(\lambda,\mathcal{A}\right) & R\left(\lambda,D_{\theta}\right) & 0\\
	0 & 0 & R\left(\lambda,D_{s}\right)
	\end{pmatrix}.
	\end{align*}
	Then $\lambda\in\rho\left(\mathcal{A}\right)$ if and only if the
	bounded operator $I-R\left(\lambda,\mathcal{A}_{0}\right)\mathcal{B}$
	is invertible, and we have
	\begin{equation}
	R\left(\lambda,\mathcal{A}\right)=\left(I-R\left(\lambda,\mathcal{A}_{0}\right)\mathcal{B}\right)^{-1}R\left(\lambda,\mathcal{A}_{0}\right).\label{eq:MainResolv}
	\end{equation}
	It is interesting to know  the conditions for which the bounded operator
	$I-R\left(\lambda,\mathcal{A}_{0}\right)\mathcal{B}$ is invertible, then  we calculate
	
	\begin{equation*}
	I-R\left(\lambda,\mathcal{A}_{0}\right)\mathcal{B}  = 	\left(\begin{array}{c:c:c}
	I & -R\left(\lambda,A\right)K & -R\left(\lambda,A\right)\delta_{0}\\ \hdashline
	0 & \left(I-e_{\lambda}R\left(\lambda,A\right)K\right) & -e_{\lambda}R\left(\lambda,A\right)\delta_{0}\\ \hdashline
	0 & -R\left(\lambda,D_{s}\right)b\left(\cdot\right)L & I
	\end{array}\right)
	\end{equation*}
	and suppose that  $I-R\left(\lambda,\mathcal{A}_{0}\right)\mathcal{B}$ is invertible, we have
	
	$$ 	\left(I-R\left(\lambda,\mathcal{A}_{0}\right)\mathcal{B}\right)^{-1}\begin{pmatrix}f\\
	g\\
	h
	\end{pmatrix}=\begin{pmatrix}x\\
	y\\
	z
	\end{pmatrix}  \iff  \left(I-R\left(\lambda,\mathcal{A}_{0}\right)\mathcal{B}\right)\begin{pmatrix}x\\
	y\\
	z
	\end{pmatrix}=\begin{pmatrix}f\\
	g\\
	h
	\end{pmatrix} $$
	\begin{align*}
		\iff &  \left(\begin{array}{c:c:c}	 I & -R\left(\lambda,A\right)K & -R\left(\lambda,A\right)\delta_{0}\\ \hdashline
			0 & \left(I-e_{\lambda}R\left(\lambda,A\right)K\right) & -e_{\lambda}R\left(\lambda,A\right)\delta_{0}\\ \hdashline
			0 & -R\left(\lambda,D_{s}\right)b\left(\cdot\right)L & I
		\end{array}\right)\begin{pmatrix}x\\
			y\\
			z
		\end{pmatrix}=\begin{pmatrix}f\\
			g\\
			h
		\end{pmatrix}\\
		 \iff &   \begin{cases}
			f=  x-R\left(\lambda,A\right)Ky-R\left(\lambda,A\right)\delta_{0}z\\
			g= \left(I-e_{\lambda}R\left(\lambda,A\right)K\right)y-e_{\lambda}R\left(\lambda,A\right)\delta_{0}z\\
			h=  -R\left(\lambda,D_{s}\right)b\left(\cdot\right)Ly+z
		\end{cases}
	\end{align*}
	Remark that $\delta_{0}R\left(\lambda,D_{s}\right)b\left(\cdot\right)L=\hat{b}\left(\lambda\right)L$
	for $ \lambda \in \mathbb{C}_{\beta} $ and from the second and third line of the above system we have 
	$$
	\left(I-e_{\lambda}R\left(\lambda,A\right)\left(K+\hat{b}\left(\lambda\right)L\right)\right)y=g+e_{\lambda}R\left(\lambda,A\right)\delta_{0}h,
	$$
	then   
	$I-R\left(\lambda,\mathcal{A}_{0}\right)\mathcal{B}$ is invertiblew if and only if $1\in\rho\left(e_{\lambda}R\left(\lambda,A\right)\left(K+\hat{b}\left(\lambda\right)L\right)\right)$
	and by a simple computation we have
	$$
	R\left(\lambda,A\right)\left(I-\left(K+\hat{b}\left(\lambda\right)L\right)e_{\lambda}R\left(\lambda,A\right)\right)^{-1}=R\left(\lambda,A+\left(K+\hat{b}\left(\lambda\right)L\right)e_{\lambda}\right).
	$$
	Then 
	$$
	\lambda\in\rho\left(\mathcal{A}\right)\text{ \ensuremath{\iff} }1\in\rho\left(e_{\lambda}R\left(\lambda,A\right)\left(K+\hat{b}\left(\lambda\right)L\right)\right)\text{ \ensuremath{\iff} }\lambda\in\rho\left(A+\left(K+\hat{b}\left(\lambda\right)L\right)e_{\lambda}\right).
	$$
	
	For simplicity of notation,  we use   $Q\left(\lambda\right)=\left(I-e_{\lambda}R\left(\lambda,A\right)\left(K+\hat{b}\left(\lambda\right)L\right)\right)^{-1}$
	and we have 
	$$
	\begin{cases}
	x =f+R\left(\lambda,A\right)Ky+R\left(\lambda,A\right)\delta_{0}z\\
	y =Q\left(\lambda\right)g+Q\left(\lambda\right)e_{\lambda}R\left(\lambda,A\right)\delta_{0}h\\
	z =R\left(\lambda,D_{s}\right)b\left(\cdot\right)Ly+h,
	\end{cases}
	$$
	then 
	\[
	\begin{cases}
	x= & f+R\left(\lambda,A\right)\left(K+\hat{b}\left(\lambda\right)L\right)Q\left(\lambda\right)g+R\left(\lambda,A\right)\delta_{0}h\\
	 & +R\left(\lambda,A\right)\left(K+\hat{b}\left(\lambda\right)L\right)Q\left(\lambda\right)e_{\lambda}R\left(\lambda,A\right)\delta_{0}h\\
	y= & Q\left(\lambda\right)g+Q\left(\lambda\right)e_{\lambda}R\left(\lambda,A\right)\delta_{0}h\\
	z= & R\left(\lambda,D_{s}\right)b\left(\cdot\right)LQ\left(\lambda\right)g+\left(R\left(\lambda,D_{s}\right)b\left(\cdot\right)LQ\left(\lambda\right)e_{\lambda}R\left(\lambda,A\right)\delta_{0}+I\right)h.
	\end{cases}
	\]
%
	Thus led to $$ \left(I-R\left(\lambda,\mathcal{A}_{0}\right)\mathcal{B}\right)^{-1}= $$
	$$
	\left(\begin{array}{c:c:c}	I & R\left(\lambda,A\right)\left(K+\hat{b}\left(\lambda\right)L\right)Q\left(\lambda\right) & \left[R\left(\lambda,A\right)\left(K+\hat{b}\left(\lambda\right)L\right)Q\left(\lambda\right)e_{\lambda}+I\right]R\left(\lambda,A\right)\delta_{0}\\ \hdashline
	0 & Q\left(\lambda\right) & Q\left(\lambda\right)e_{\lambda}R\left(\lambda,A\right)\delta_{0}\\ \hdashline
	0 & R\left(\lambda,D_{s}\right)b\left(\cdot\right)LQ\left(\lambda\right) & \left(R\left(\lambda,D_{s}\right)b\left(\cdot\right)LQ\left(\lambda\right)e_{\lambda}R\left(\lambda,A\right)\delta_{0}+I\right)
	\end{array}\right),	$$
	
	and from (\ref{eq:MainResolv}) one has 
	$$
	R\left(\lambda,\mathcal{A}\right)=\begin{pmatrix}R_{11} & R_{12} & R_{13}\\
	R_{21} & R_{22} & R_{23}\\
	R_{31} & R_{32} & R_{33}
	\end{pmatrix}
	$$
	with 
	$$
	\begin{cases}
	R_{11}= & R\left(\lambda,A\right)+R\left(\lambda,A\right)\left(K+\hat{b}\left(\lambda\right)L\right)Q\left(\lambda\right)e_{\lambda}R\left(\lambda,\mathcal{A}\right)\\
	R_{12}= & R\left(\lambda,A\right)\left(K+\hat{b}\left(\lambda\right)L\right)Q\left(\lambda\right)R\left(\lambda,D_{\theta}\right)\\
	R_{13}= & \left[R\left(\lambda,A\right)\left(K+\hat{b}\left(\lambda\right)L\right)Q\left(\lambda\right)e_{\lambda}+I\right]R\left(\lambda,A\right)\delta_{0}R\left(\lambda,D_{s}\right)\\
	R_{21}= & Q\left(\lambda\right)e_{\lambda}R\left(\lambda,\mathcal{A}\right)\\
	R_{22}= & Q\left(\lambda\right)R\left(\lambda,D_{\theta}\right)\\
	R_{23}= & Q\left(\lambda\right)e_{\lambda}R\left(\lambda,A\right)\delta_{0}R\left(\lambda,D_{s}\right)\\
	R_{31}= & R\left(\lambda,D_{s}\right)b\left(\cdot\right)LQ\left(\lambda\right)e_{\lambda}R\left(\lambda,A\right)R\left(\lambda,D_{s}\right)\\
	R_{32}= & R\left(\lambda,D_{s}\right)b\left(\cdot\right)LQ\left(\lambda\right)R\left(\lambda,D_{\theta}\right)\\
	R_{33}= & \left(R\left(\lambda,D_{s}\right)b\left(\cdot\right)LQ\left(\lambda\right)e_{\lambda}R\left(\lambda,A\right)\delta_{0}+I\right)R\left(\lambda,D_{s}\right).
	\end{cases}
	$$
	A simple computation gives the desired resolvent operator components
	\begin{align*}
		R_{11}  = & R\left(\lambda,A\right)+R\left(\lambda,A\right)\left(K+\hat{b}\left(\lambda\right)L\right)\left(I-e_{\lambda}R\left(\lambda,A\right)\left(K+\hat{b}\left(\lambda\right)L\right)\right)^{-1}e_{\lambda}R\left(\lambda,\mathcal{A}\right)\\
		 = & R\left(\lambda,A\right)\left(I-\left(K+\hat{b}\left(\lambda\right)L\right)e_{\lambda}R\left(\lambda,A\right)\right)^{-1}\\
		 = & R\left(\lambda,A+\left(K+\hat{b}\left(\lambda\right)L\right)e_{\lambda}\right),\\
		R_{12}  = & R\left(\lambda,A\right)\left(K+\hat{b}\left(\lambda\right)L\right)\left(I-e_{\lambda}R\left(\lambda,A\right)\left(K+\hat{b}\left(\lambda\right)L\right)\right)^{-1}R\left(\lambda,D_{\theta}\right)\\
		 = & R\left(\lambda,A\right)\left(I-\left(K+\hat{b}\left(\lambda\right)L\right)e_{\lambda}R\left(\lambda,A\right)\right)^{-1}\left(K+\hat{b}\left(\lambda\right)L\right)R\left(\lambda,D_{\theta}\right)\\
		 = & R\left(\lambda,A+\left(K+\hat{b}\left(\lambda\right)L\right)e_{\lambda}\right)\left(K+\hat{b}\left(\lambda\right)L\right)R\left(\lambda,D_{\theta}\right),\\
		R_{13}  = & \left[R\left(\lambda,A\right)\left(K+\hat{b}\left(\lambda\right)L\right)\left(I-e_{\lambda}R\left(\lambda,A\right)\left(K+\hat{b}\left(\lambda\right)L\right)\right)^{-1}e_{\lambda}+I\right]R\left(\lambda,A\right)\delta_{0}R\left(\lambda,D_{s}\right)\\
		 = & \left[R\left(\lambda,A\right)\left(I-\left(K+\hat{b}\left(\lambda\right)L\right)e_{\lambda}R\left(\lambda,A\right)\right)^{-1}\left(K+\hat{b}\left(\lambda\right)L\right)e_{\lambda}+I\right]R\left(\lambda,A\right)\delta_{0}R\left(\lambda,D_{s}\right)\\
		 = & \left[R\left(\lambda,A+\left(K+\hat{b}\left(\lambda\right)L\right)e_{\lambda}\right)\left(K+\hat{b}\left(\lambda\right)L\right)e_{\lambda}+I\right]R\left(\lambda,A\right)\delta_{0}R\left(\lambda,D_{s}\right),\\
		R_{21}  = & \left(I-e_{\lambda}R\left(\lambda,A\right)\left(K+\hat{b}\left(\lambda\right)L\right)\right)^{-1}e_{\lambda}R\left(\lambda,\mathcal{A}\right)\\
		 = & e_{\lambda}R\left(\lambda,\mathcal{A}\right)\left(I-\left(K+\hat{b}\left(\lambda\right)L\right)e_{\lambda}R\left(\lambda,A\right)\right)^{-1}\\
		 = & e_{\lambda}R\left(\lambda,A+\left(K+\hat{b}\left(\lambda\right)L\right)e_{\lambda}\right),
	\end{align*}
	
	\begin{align*}
		R_{22}  = & \left(I-e_{\lambda}R\left(\lambda,A\right)\left(K+\hat{b}\left(\lambda\right)L\right)\right)^{-1}R\left(\lambda,D_{\theta}\right)\\
		 = & \left[I+e_{\lambda}R\left(\lambda,A\right)\left(I-\left(K+\hat{b}\left(\lambda\right)L\right)e_{\lambda}R\left(\lambda,A\right)\right)^{-1}\left(K+\hat{b}\left(\lambda\right)L\right)\right]R\left(\lambda,D_{\theta}\right)\\
		 = & \left[I+e_{\lambda}R\left(\lambda,A\right)\left(I-\left(K+\hat{b}\left(\lambda\right)L\right)e_{\lambda}R\left(\lambda,A\right)\right)^{-1}\left(K+\hat{b}\left(\lambda\right)L\right)\right]R\left(\lambda,D_{\theta}\right)\\
		 = & \left[I+e_{\lambda}R\left(\lambda,A+\left(K+\hat{b}\left(\lambda\right)L\right)e_{\lambda}\right)\left(K+\hat{b}\left(\lambda\right)L\right)\right]R\left(\lambda,D_{\theta}\right),\\
		R_{23}  = & R_{21}\delta_{0}R\left(\lambda,D_{s}\right)\\
		 = & e_{\lambda}R\left(\lambda,A+\left(K+\hat{b}\left(\lambda\right)L\right)e_{\lambda}\right)\delta_{0}R\left(\lambda,D_{s}\right),\\
		R_{31}  = & R\left(\lambda,D_{s}\right)b\left(\cdot\right)LR_{23}\\
		 = & R\left(\lambda,D_{s}\right)b\left(\cdot\right)Le_{\lambda}R\left(\lambda,A+\left(K+\hat{b}\left(\lambda\right)L\right)e_{\lambda}\right)R\left(\lambda,D_{s}\right),\\
		R_{32}  = & R\left(\lambda,D_{s}\right)b\left(\cdot\right)LR_{22}\\
		 = & R\left(\lambda,D_{s}\right)b\left(\cdot\right)L\left[I+e_{\lambda}R\left(\lambda,A+\left(K+\hat{b}\left(\lambda\right)L\right)e_{\lambda}\right)\left(K+\hat{b}\left(\lambda\right)L\right)\right]R\left(\lambda,D_{\theta}\right),
	\end{align*}
	and
	\begin{align*}
		R_{33}  = & \left(R\left(\lambda,D_{s}\right)b\left(\cdot\right)LR_{21}\delta_{0}+I\right)R\left(\lambda,D_{s}\right)\\
		 = & \left(R\left(\lambda,D_{s}\right)b\left(\cdot\right)Le_{\lambda}R\left(\lambda,A+\left(K+\hat{b}\left(\lambda\right)L\right)e_{\lambda}\right)\delta_{0}+I\right)R\left(\lambda,D_{s}\right)
	\end{align*}
	which complete the proof. 
\end{proof}

\begin{remark}
If $b=0$ we have 	
	\[
	R\left(\lambda,\mathcal{A}\right)=\left(\begin{array}{ccc}
	R_{11} & R_{12} & R_{13}\\
	 R_{21} & R_{22} & R_{23}\\
	 R_{31} & \begin{array}{c}
	R_{32}\end{array} & \begin{array}{c}
	R_{33}\end{array}
	\end{array}\right)
	\]
	with 
	\[
	\begin{cases}
	R_{11}=R\left(\lambda,A+K_{\lambda}\right)\\
	R_{12}=R\left(\lambda,A+K_{\lambda}\right)KR\left(\lambda,D_{\theta}\right)\\
	R_{13}=\begin{array}{c}
	\left[R\left(\lambda,A+K_{\lambda}\right)K_{\lambda}+I\right]R\left(\lambda,A\right)\delta_{0}R\left(\lambda,D_{s}\right)\end{array}\\
	R_{21}=e_{\lambda}R\left(\lambda,A+K_{\lambda}\right)\\
	R_{22}=\begin{array}{c}
	\left[e_{\lambda}R\left(\lambda,A+K_{\lambda}\right)K+I\right]R\left(\lambda,D_{\theta}\right)\end{array}\\
	R_{23}=e_{\lambda}R\left(\lambda,A+K_{\lambda}\right)K_{\lambda}\delta_{0}R\left(\lambda,D_{s}\right)\\
	R_{31}=\begin{array}{c}
	R_{32}\end{array}=0\\
	R_{33}=\begin{array}{c}
	R\left(\lambda,D_{s}\right)\end{array}
	\end{cases}
	\]

	As $b\left(\cdot\right)=0$,   the function $g$ defined in (\ref{eq:Fct_g})
	is equal to $0$, then if we replace the third space $L^{p}\left(\mathbb{R}_{+},X\right)$
	by $\left\{ 0\right\} $, we obtain the same result as in \cite[Proposition 3.19, page 56]{batkai2005semigroups}
	$$
	R\left(\lambda,\mathcal{A}\right)\begin{pmatrix}x\\
	\varphi\\
	0
	\end{pmatrix}=\left(\begin{array}{c:c}
	R\left(\lambda,A+K_{\lambda}\right) & R\left(\lambda,A+K_{\lambda}\right)KR\left(\lambda,D_{\theta}\right)\\
	\hdashline e_{\lambda}R\left(\lambda,A+K_{\lambda}\right) & \begin{array}{c}
	\left[e_{\lambda}R\left(\lambda,A+K_{\lambda}\right)K+I\right]R\left(\lambda,D_{\theta}\right)\end{array}
	\end{array}\right)\begin{pmatrix}x\\
	\varphi
	\end{pmatrix}\times\left\{ 0\right\} .
	$$
\end{remark}

\section{ Application.}  
In order to apply the results obtained in Sections 2 and 3, we consider the two following examples.
\begin{example}
	Consider the following diffusion equation. Let $t\geq0$ and $ x\in\mathbb{R} $
	\begin{equation}
	\begin{cases}
	\frac{\partial}{\partial t}z(x,t)=\Delta z(x,t)+\int_{0}^{t}e^{-\left(t-s\right)}\int_{-1}^{0}d\mu\left(\theta\right)z(x,s+\theta)d\theta ds+z(x,t-1); & \\
	z(x,0)=z^{0}(x)\\
	z(x,\theta)=\varphi(x,\theta); \quad \theta \in[-1,0].&
	\end{cases}\label{eq:Examp1}
	\end{equation}
	In order to write the system (\ref{eq:Examp1}) as the abstract form
	of system \eqref{RIE}, we take
	\begin{itemize}
		\item the state space $X=L^{2}(\mathbb{R})$,
		\item the operator $A=\Delta$, with $D(A)=W^{2,2}\left(\mathbb{R}\right)$,
		\item the function $\mathbb{R}_{+}\ni t\mapsto z(t)=z(\cdot,t)\in L^{2}(\mathbb{R})$,
		and the history function $z_{t}:[-1,0]\rightarrow L^{2}(\mathbb{R})$; $z_{t}(s)=z(t+s)$,
		\item the state delay operator $L:W^{1,2}\left([-1,0],L^{2}(\mathbb{R})\right)\rightarrow L^{2}(\mathbb{R})$,
		defined by $L\varphi=\int_{-1}^{0}d\mu\left(\theta\right)\varphi(\theta)$,
		with $\mu:[-1,0]\rightarrow X$ be of bounded variation.
		\item the state delay operator $K:W^{1,2}\left([-1,0],L^{2}(\mathbb{R})\right)\rightarrow L^{2}(\mathbb{R})$,
		defined by $K=\delta_{-1}$. 
	\end{itemize}
	It is well known that $A$ generates an analytic $C_{0}$-semigroup
on $X$, denoted by $\mathbb{T}(t)$ called the diffusion semigroup
(see \cite[Chapter II, page 69]{EngelNagel;2000}. Moreover, $\mathbb{T}(t)$
is expressed as follows 
\[
(\mathbb{T}(t)f)(s):=\begin{cases}
(4\pi t)^{-1/2}\int_{\mathbb{R}}\mathrm{e}^{-(s-\theta)^{2}/4t}f(\theta)\mathrm{d}\theta & t>0\\
f(s) & t=0
\end{cases};\,\text{ where }s\in\mathbb{R}.
\]
From \eqref{cond.miyadera} we have already mentioned that  condition $\left(M_{1}\right)$ is satisfied by $L$.
It is well-known  see \cite[Theorem 6.13, page 74]{A.Pazy;1983} that $\left(-\Delta\right)^{\gamma}$ satisfied  the following estimation $
\left\Vert t^{\gamma}\left(-\Delta\right)^{\gamma}\mathbb{T}(t)\right\Vert \leq M, 
$ with $ M $ and $ \gamma $ are positive constants. 
By a simple computation with $ \gamma\in (0,\frac{1}{2}) $  we have :\\

$\displaystyle\int_{0}^{\alpha}\left\Vert \left(-\Delta\right)^{\gamma}\mathbb{T}(t)x\right\Vert ^{2}dt  \leq M^{2}\frac{1}{1-2\gamma}\alpha^{1-2\gamma}\left\Vert x\right\Vert ^{2}$, which implies that $\mathcal{F}$ is a $ 2 $-admissible observation operator for $ (\mathbb{T}(t)) $. Then  the well-posedness of the system (\ref{eq:Examp1}) is reached by Theorem \ref{thm:Wellposedness}.
\end{example}

\begin{example}
	Consider the following retarded heat equation with the Dirichlet boundary condition  on  the state space $X=L^{2}[0,\pi]$, for $ x\in[0,\pi] $ and $ t\geq 0 $
	\[
	\begin{cases}
	\frac{\partial}{\partial t}z(x,t)=\Delta z(x,t)+\int_{0}^{t}e^{-\left(t-s\right)}z(x,s-1)ds+z(x,t-1);\\
	z(0,t)=z(\pi,t)=0;\\
	z(x,0)=z^{0}(x);\\
	z(x,s)=\varphi(x,s);\quad s\in[-1,0].
	\end{cases}
	\]
	
	Here we take
	\begin{itemize}
		\item the operator $A=\Delta$, with $D(A)= \left\lbrace f\in H^2[0,\pi] \, : \, f(0)=f(\pi)=0 \right\rbrace $,
	 
		\item the state delay operators $L,K:H^{1}\left([-1,0],X\right)\rightarrow X$,
		defined by $L=K=\delta_{-1}$, 
		\item the scalar kernel is defined by  $b\left(t\right)=e^{- t}$.
	\end{itemize}
	The well-posedness of this system is due to the fact that $ A $ generates a  $ C_0 $-semigroup and from \eqref{eq:myExmple}, we have already mentioned that condition $\left(M_{2}\right)$ is satisfied by $L$, these conditions verify Theorem \ref{thm:Wellposedness}. \\
	Moreover, $ A $ generates a  compact $ C_0 $-semigroup $ \mathbb T (t) $ on $ X $ (see, e.g., \cite[Example 1.1]{G.Weiss;1989b}, \cite[Example 2.6.8]{TucsnakWeiss2009} and \cite[Example 2.1.1]{Curtain.Zwart;1995}), the compactness of $ \mathbb T (t) $ guarantees that
	$ \sigma (A) =\sigma_p (A)  $. 
	In the following orthonormal basis $\phi_{n}(x)=\sqrt{\frac{2}{\pi}}\cos(nx)$
	for $n\geq1$,  we have 
	$ \sigma _p(A) = \left\lbrace -n^2   :  n = 1, 2, \cdots \right\rbrace$. Then from Theorem \ref{thm:Spectr}, one has for   
	$\lambda\notin\mathbb{C}_{0}\cap\rho(A)$ :
	\begin{alignat*}{1}
	\lambda\in\sigma\left(\mathcal{A}\right) & \iff\lambda\in\sigma\left(A+\delta_{-1}e_{\lambda}+\hat{b}\left(\lambda\right)\delta_{-1}e_{\lambda}\right)\\
	& \iff\left(\lambda-\delta_{-1}e_{\lambda}-\hat{b}\left(\lambda\right)\delta_{-1}e_{\lambda}\right)\in\sigma\left(A\right)
	\end{alignat*}
	
	then 
	\begin{align*}\sigma\left(\mathcal{A}\right) & =\left\{ \lambda\notin\mathbb{C}_{0}\,\mid\,\lambda-\frac{2+\lambda}{1+\lambda}e^{-\lambda}=-n^2\right\} \\
	& =\left\{ \lambda\notin\mathbb{C}_{0}\,\mid\,\left(\lambda+n^2\right)\left(\lambda +1\right)=(2+\lambda)e^{-\lambda}\right\} .
	\end{align*}
	
\end{example}


\medskip

\medskip


\end{document}